\documentclass[11pt]{amsart}
\usepackage{amsfonts}
\usepackage{amscd}
\usepackage{amssymb}
\usepackage{graphics}
\usepackage{amsmath}
\usepackage{hyperref}
\usepackage{breqn}
\usepackage{multicol}
\usepackage{pdfpages}
\usepackage{import}
\usepackage{svg}
\newtheorem{theorem}{Theorem}[section]
\newtheorem{corollary}[theorem]{Corollary}
\newtheorem{claim}{Claim}

\newtheorem{proposition}[theorem]{Proposition}
\newtheorem{definition}[theorem]{Definition}
\newtheorem{lemma}[theorem]{Lemma}

\newcommand{\iprod}[1]{\langle#1\rangle}

\newcommand{\gBinom}{\genfrac{[}{]}{0pt}{}}

\begin{document}
	\author{Paul Vollrath}
	\address{Mathematics and Computer Science\\
		Weizmann Institute of Science\\
		Rechovot 76100, Israel\\}
	\email{paul.vollrath@weizmann.ac.il}
	\date{\today}
    \title{Full Resolution to Papikian's Conjecture}
    
    \begin{abstract}
        We prove Papikian's conjecture on the spectrum of the signed up-down walk on the spherical building. Namely, we show that for the operator $\Delta_i^+=\delta_id_i$ in the spherical building $X_{n-2,q}$ of dimension $n-2$ and thickness $q+1$, the number of distinct eigenvalues is independent of $q$ and for $q\rightarrow \infty$ the positive eigenvalues converge to $n-1,\dots, n-i$.
    \end{abstract}
    \maketitle
    \section{Introduction}
    High dimensional expansion  has been a very active area of research. The foundational example of HDX was constructed as \cite{LSV05} as quotients of buildings. In these HDXs each link is a \textbf{spherical building} (which we will explicitly define in Section \ref{sec:SphericalBuilding}). Various types of expansion in these complexes has been shown via a local to global approach, i.e. extrapolating expansion in the spherical building to the complex as a whole (c.f. \cite{EK17}), which makes understanding the expansion of spherical buildings very relevant. \\
    Here we define high dimensional expansion via the spectral gap in the \textbf{up-down} walk at each level. In the up-down walk on a simplicial complex $X$ one walks on the $i$-simplices $X(i)$ via the $i+1$-simplices. Specifically starting at an $i$-simplex $s$ one picks a random $i+1$-simplex $t\in X(i+1)$ which contains $s$ and then walks to a random $i$-simplex $s'\subset t$. For this walk  it was shown in a sequence of work (c.f. \cite{DDFH18}, \cite{KO21b}, \cite{Eigen23}) that in the more general setting of posets, the up-down walk exhibits the phenomenon of \textbf{eigenstripping}. Applied to an $n$-dimensional spherical building  this means that for the up-down walk,  the entire spectrum of the up-down walk decomposes into $n+1$ distinct strip where the position of each band depends on a generalized regularity notion.\\
    The up-down walk as defined above has a strong homological flavour: Namely taking the up-step from $X(i)$ to $X(i+1)$ is comparable to the co-boundary map $d$ (up to a sign) and the down step is in turn comparable to the boundary map $\delta$. Motivated by this observation one may think of the composed operator $\Delta^+=\delta \circ d:X(i)\rightarrow X(i)$ as a homological analogue to the (unsigned) up-down walk. For this reason we will refer to the operator $\delta\circ d$ as the \textbf{signed} up-down walk operator. From a slightly different perspective M. Papikian conjectured in \cite{Pap16} that also for the signed up-down walk operator the spectrum decomposes into several strips.  More precisely he conjectured 
    \begin{theorem}
        For the signed up down walk on an $n$-dimensional building of type $A$ at level $i$ the number of distinct eigenvalues is bounded by the expression 
        \begin{equation}
            \sum_{k=0}^{i+1}\binom{i+2}{k} (-1)^k (i+2-k)^n
        \end{equation}
        which is independent of the thickness $q+1$. Further, when fixing the dimension $n$ and for the thickness going to infinity, the non-zero eigenvalues of the signed up-down walk approach $n+1,\dots, n+1-i$ at level $i$. 
    \end{theorem}
    (in the proof split into  Proposition \ref{thm:NumberEigenvalues} and Proposition \ref{thm:Eigenstripping}) 
    A. Lew showed the conjecture in \cite{Lew23} for the level $i=0$ i.e. the Laplacian of the $1$-skeleton of the spherical building. Our methods differ significantly, we use the group of automorphisms of the spherical building to retract each of the walks to what is essentially an apartment of the building (Sec. \ref{sec:TheQuotient}). More precisely, we reinterpret the up-down walk of level $i$ as a walk on a weighted graph $X^i$ which has the $i$-simplices as vertices and the edge weights are induced by the walk matrix. Here both the adjacency structure and the walk matrix are derived from the geometry of the building and thus the group of automorphisms of the building also acts by automorphisms on the graphs $X^i$.\\ 
    The subgroup central to our analysis is, the \textbf{stabilizer} $H$ of a top-level simplex, which is the group of upper triangular matrices. For each of the graphs $X^i$ we show that the walk on $X^i$ and induced walk on the quotient $X^i/H$ have the same spectrum. This quotient map has a uniform description for all $X^i$ via the \textbf{height-profile map} (Definition \ref{def:Heights}). The height of a vector is the largest index of a non-zero coefficient in a canonical basis and the height profile map associates to a vector spaces the collection of vector heights featuring in the vector space. In this manner the vertices of the quotient $X^i/H$ can be canonically identified with $i$-flags of subsets of the interval $[n]$.\\   Since the number of such flags only depends on $n$ and not $q$, this proves that the number of distinct eigenvalues of the up down walk on level $i$ is bounded in terms of $n$ i.e. the first part of Papikian's conjecture. For the second part of the conjecture, we show that when driving the thickness to infinity the resulting walk on the $i$-flags of $[n]$ decomposes the building into several copies of complete complexes with dimension up to $n$ (c.f. Section \ref{sec:TheBlock}). Analyzing the walk on each of these simplices independently then results in the second part of the conjecture.\\
    Since the basic approach is viable for all spherical buildings (independent of their type) we intend to generalize this result to all spherical buildings in future work. 

    \subsection*{Organization} In section 
    \section{Basics}
    \subsection{Simplicial Complexes \& the Signed Up-Down Walk}
    A finite \textbf{simplicial complex} $X$ is a collection of subsets of a finite set $[n]=\{1,\dots, n\}$  which is closed under inclusion, meaning if $s\in X$ then any subset $t\subset s$ is also contained in $X$. Each element in $X$ is referred to as a \textbf{simplex}. The dimension of a simplex $s\in X$ is $|s|-1$ and one denotes the set of $i$-dimensional simplices of $X$ as $X(i)$. The dimension of the whole simplicial complex is the dimension of the largest simplex it contains. A simplicial complex is called \textbf{pure} if every simplex $s\in X$ is contained in a simplex of  dimension $dim(X)$ i.e. a top-dimensional simplex. All simplicial complexes we consider are pure.\\
    The set of functions $C^i(X):X(i)\rightarrow \mathbb{R}$ are called $i$-chains. In order to define both the \textbf{boundary} and \textbf{co-boundary} maps one needs to introduce orientation and a weight function. Defining the orientation one first enumerates the vertices of $X$. This means one fixes an embedding of $X(0)$ into $[|X(0)|]$ and maps $X(i)$ accordingly to subsets of size $i+1$ of $[|X(0)|]$. Fixing such an \textbf{enumeration} (the specific enumeration is arbitrary) one identifies $x\in X(0)$ with its label under this enumeration. For two vertices $x_1,x_2$ the expression $x_1<x_2$ indicates that the label of $x_1$ is less than the label of $x_2$. Let $x_0,\dots, x_{i+1}$ be the vertices of the simplex $s\in X$ such that $x_0<\dots<x_{i+1}$ and $f\in C^i(X)$ be an $i$-chain. Then the co-boundary of $f$ is denoted by $df\in C^{i+1}(X)$ and evaluated at $s$ $df$ is given by:
    \begin{equation}
        d_if(s)=\sum_{k=0}^{i+1} (-1)^k f([x_0,\dots, \hat{x}_k,\dots, x_{i+1}])
        \label{eq:CoBoundaryOperator}
    \end{equation}
    where $[x_0,\dots, \hat{x}_k,\dots, x_{i+1}]$ denotes the simplex $t\subset s$ which is missing the vertex $x_k$. In this manner the enumeration induces the co-boundary map $d_i: C^i\rightarrow C^{i+1}$. In order to define the boundary map one introduces a weight function:
    \begin{definition}
        Let $X$ be a pure $n$-dimensional simplex. For a simplex $s\in X(i)$ let its weight $w(s)$ be given by:
        \begin{equation}
            w(s)=|\{t\in X(n)|s\subset t\}|
        \end{equation}
        i.e. the weight of $s$ is given by the number of top-dimensional simplices $s$ is contained in.
        \label{def:Weight}
    \end{definition}
    This weight function on the simplices induces a scalar product on the space of co-chains $C^i$ by setting
    \begin{equation}
        \iprod{f,g}=\sum_{s\in X} w(s) f(s)g(s)
    \end{equation}
    for two $i$-chains $f,g\in C^i(X)$. Using this scalar product on the $C^i$ the boundary map $\delta_i:C^{i+1}\rightarrow C^i$ is the adjoint map of $d_i$ namely for all $f\in C^i$ and $g\in C^{i+1}$ holds $\iprod{d_i f,g}=\iprod{f,\delta_i g}$. Explicitly, this means that for $s\in X(i)$ and $f\in C^{i+1}$ its boundary evaluated at $s$ reads (c.f. \cite{Pap16}):
    \begin{equation}
        \delta_i f(s)=\sum_{x\in X(0)|x\cup s\in X(i+1)}\frac{w(x\cup s)}{w(s)} sign(x\cup s)f(x\cup s)
        \label{eq:BoundaryOperator}
    \end{equation}
    Where $sign(x\cup s)$ is the sign of the permutation ordering the $i+2$-tuple $x,x_0,\dots x_i$ according to their labels while assuming that $(x_0<\dots<x_i)=s$.
    \begin{definition}
        The operator $\Delta_i^+=\delta_id_i$ on chains $C^i$ is called the \textbf{signed up-down walk} on $i$-simplices. In other contexts it might be referred to as the curvature transformation (e.g. \cite{Pap16}). 
    \end{definition}
    Clearly $\Delta_i^+$ is self adjoint under the scalar product induced by the weight function and positive since $\iprod{\Delta_i^+f,f}=\iprod{d_if,d_if}\geq 0$. Further, the spectrum of $\Delta_i^+$ does not depend on the specific enumeration (c.f. \cite{Hat02} chapter 2).\\
    Before, getting into the weeds of discussing the spectrum of the signed up-down walk on the spherical building a quick explanation for the nomenclature is in order. Clearly, as defined $\Delta_i^+$ is not a random walk for two reasons: Obviously, some entries of $\Delta_i^+$ are negative, therefore we refer to it as a signed walk. The other issue is that $\sum_{x\in X(0)|x\cup s\in X(i+1)}\frac{w(x\cup s)}{w(s)}\neq 1$ in general i.e. $\Delta_i^+$ is not stochastic even if one ignores the signs. But on second thought one realizes that in a pure simplicial complex of dimension $\sum_{x\in X(0)|x\cup s\in X(i+1)}\frac{w(x\cup s)}{w(s)}=n-i$. This can be seen as follows: For any $i$-simplex $s$ and any top level simplex $s\subset t$ has $n+1$ vertices. This means that there are $n-i$ intermediate $i+1$-simplices $s\subset s'\subset t$. Thus one has:
    \begin{equation}
        \sum_{u\in X(i)}|\Delta_i^+(s,u)|=\frac{\sum_{s'\in X(i+1)|s\subset s'}w(s')}{w(s)}=(n-i)\frac{\sum_{t\in X(n)}1}{w(s)}=n-i
    \end{equation}
    For this reason we deem calling $\Delta_i^+$ a signed walk in a pure complex appropriate.

    \subsection{The Spherical Building}
    \label{sec:SphericalBuilding}
     The simplicial complex of interest is the spherical building $X_{n-2,q}$ of type $A_n$. The spherical building $X_{n-2,q}$ is the flag complex associated to the vector space $\mathbb{F}_q^n$ for a prime-power $q$ or explicitly one defines:
	\begin{definition}
		Fix $n$ and a prime power $q$ and let $\mathbb{F}_q$ be the field of size $q$.The spherical building $X_{n-2,q}$ has as vertex set all non-trivial linear subspaces of $\mathbb{F}_q^n$ and two vertices $x,y\in X_{n-2,q}(0)$ are connected by an edge if the associated subspaces $\mathbf{x},\mathbf{y}$ fulfill an inclusion relation i.e. $\mathbf{x}\subset \mathbf{y}$ or $\mathbf{y}\subset \mathbf{x}$. Making this complex clique, i.e. filling in all $k$-cliques in $X_{n-2,q}$ by a $k-1$-simplex for  $2\leq k\leq n-2$ one obtains the spherical building $X_{n-2,q}$.
		\label{def:SphericalBuilding}
    \end{definition}
    As two distinct subspaces $\mathbf{x}, \mathbf{y}\subset \mathbb{F}_{n,q}$ of the same dimension cannot fulfill an inclusion relation, there is no edge between the corresponding vertices $x,y\in X_{n,q}(0)$. Thus, assigning $dim(\mathbf{x})$ as a label to vertex $x$ provides a legal coloring of the building. From now on, for $0<k<n$ we will refer to the set $\{x\in X_{n-2,q}|dim(\mathbf{x})=k\}$ as the color-$k$ vertices. \\ 
    The group $G=GL_n(\mathbb{F}_q)$ acts naturally on the spherical building $X_{n-2,q}$ via its action on the underlying vector space $\mathbb{F}_q^n$. Namely, for $g\in G$ and vertex $x\in X_{n-2,q}$ with associated vector space $\mathbf{x}\subset \mathbb{F}_q^n$ the vertex $gx$ is associated to the subspace $g\mathbf{x}\subset \mathbb{F}_q^n$. Since the action of $G$ on $\mathbb{F}_q^n$ preserves inclusion relations of subspaces, $g\in G$ preserves adjacency relations when acting on the vertices of $X_{n-2,q}$. This means that $g\in G$ is an automorphism of the one skeleton of the building $X_{n-2,q}$ but as the building is a clique complex, i.e. its simplicial structure is defined by its one skeleton, $g\in G$ is actually an automorphism of the entire building.\\
    For the further discussion we need some transitivity results regarding the action of $G$ on the building:
    \begin{proposition}
        (c.f. \cite{Brown89}) The action of $G$ is transitive on the top-level cells of the spherical building.
    \end{proposition}
    As a lemma on can derive an even stronger result and to that end one may define the type of a simplex. For simplex $s=\{x_0,\dots, x_i\} $ its type is $\{col(x_0),\dots,col(x_i)\}$.
    \begin{lemma}
        $G$ acts transitively on the simplices of type $\{c_0,\dots, c_i\}$ for all possible types.
        \label{lem:MasterTransitivity}
    \end{lemma}
    \begin{proof}
        Fix a top-level simplex $s$, note that the type of this simplex is $\{1, \dots, n-1\}$ for type $\{c_0,\dots, c_i\}$ it contains exactly one simplex $t$ of this type. Let $t'\in X_{n,q}$ be another simplex of type $\{c_0,\dots, c_i\}$. The spherical building is pure so $t'$ is contained in a top-level simplex $s'$. Since $G$ acts transitively on top-level simplices there is $g$ such that $gs'=s$. As $G$ preserves the color of vertices it also preserves types of simplices so $gt'\subset s$ must be $t$ as $t$ is the unique simplex of type $\{c_0,\dots, c_i\}$ in $s$.
    \end{proof}
    By acting on the building itself, the group acts on the space of co-chains as well, namely for $g\in G$, $f\in C^i(X_{n-2,q})$ and $i$-simplex $s\in C^i(X_{n-2,q})$ one has:
    \begin{equation}
        gf(s)=f(g^{-1}s)
    \end{equation}
    Returning to discussing the signed up-down walk on the spherical building, we fix an enumeration of the vertices according to color meaning for two vertices $x,y\in X_{n-2,q}(0)$ the condition $col(x)<col(y)$ implies $x<y$ holds. 
    \begin{proposition}
        Given the fixed enumeration the co-boundary map is $G$-equivariant, meaning that for $g\in G$ and $f\in C^i(X_{n-2,q})$ 
        \begin{equation}
            d_igf=gd_if
        \end{equation}
    \end{proposition}
    \begin{proof}
        The proof is a simple calculation. For $s\in X_{n-2,q}(i+1)$ one has
        \begin{dmath}
             d_i gf(s)=d_if(g^{-1}s)=\sum_{k=0}^{i+1} (-1)^k f([g^{-1}x_0,\dots, g^{-1}\hat{x}_k,\dots, g^{-1}x_{i+1}])=\sum_{k=0}^{i+1} (-1)^k gf([x_0,\dots,\hat{x}_k,\dots, x_{i+1}])=gd_if(s)
        \end{dmath}
        where the third equality relied on the fact that each of the vertices of $s$ have a different color, the action of $g$ does not change the color of the vertices and by assumption on the enumeration this implies that if $x_0,\dots, x_{i+1}$ is ordered according to the enumeration then so is $g^{-1}x_0,\dots,g^{-1}x_{i+1}$.
    \end{proof}
    In order to show that also the boundary map is $G$-equivariant one notes that with respect to the standard basis of $C^i$ (i.e. indicators of individual simplices) any element of $G$ is represented by a permutation. Further, the weight of a simplex depends only on how many top-level simplices it is contained in. Since $G$ is a group of automorphisms i.e. $s\in X_{n-2,q}(i)$ and $gs$ must be contained in the same number of top-level simplices, the weight function is $G$ invariant. Together this means that the adjoint of $g\in G$ with respect to the scalar product induced by the weight function is $g^{-1}$.
    \begin{proposition}
        Given the fixed enumeration, also the boundary map is $G$-equivariant i.e.
        \begin{equation}
            \delta_ig=g\delta_i
        \end{equation}.
    \end{proposition}
    \begin{proof}
        Since the boundary map is the adjoint map of the co-boundary map, the equi-variance of the co-boundary map implies the equi-variance of the boundary map as follows: For every $f\in C^i(X_{n-2,q})$, $h\in C^{i+1}(X_{n-2,q})$ and $g\in G$ holds:
        \begin{equation}
            \iprod{f,\delta_i gh}=\iprod{d_if,gh}=\iprod{g^{-1}d_if,h}=\iprod{d_ig^{-1}f,h}=\iprod{g^{-1}f,\delta_ih}=\iprod{f,g\delta_ih}
        \end{equation}
        Since this equality holds for any $f$, $h$ and $g$, and the scalar product is non-degenerate one has $\delta_igh=g\delta_ih$ i.e. co-variance of the boundary map.
    \end{proof}
    Combining the $G$-equivariance of the boundary and the coboundary map, one has
    \begin{equation}
        g^{-1}\Delta_i^+g=\Delta_i^+    
    \end{equation}
    i.e. $\Delta_i^+$ is $G$-invariant, positive and self adjoint.

    \subsection{Invariance of $\Delta_i^+$ under Automorphisms}
    \label{sec:WalkInvariance}
    Showing that the up-down walk is invariant under $G$ in the spherical building relied on finding a well-behaved enumeration of its vertices. We will show now that in any simplicial complex $X$ the signed up-down walk is invariant under its group of automorphisms $Aut(X)$. This will be useful for calculating the spectrum of the signed-up down walk on the simplex in Section \ref{sec:SimplexSpecs}.\\
    Towards this goal one needs to introduce some technical definitions: For a permutation $\sigma \in S_n$ and a subset $v=\{v_0,\dots,v_i\}\subset [n]$ assuming that $v_0<\dots<v_i$ let $sign(\sigma\restriction_v)$ be the sign of the permutation ordering the $i+1$-tuple $\sigma(v_0),\dots, \sigma(v_i)$. Denote $\hat{v}_k$ be $v$ without the $v_k$ i.e. $\hat{v}_k=\{v_0,\dots,v_{k-1},v_{k+1},\dots, v_i\}$. \\
    Now let $g\in Aut(X)$ be an automorphism of the simplicial complex $X$. We may view it as a permutation of the vertices $x\in X(0)$ and have it act as follows on the $i$-chains: For a simplex $s=\{x_0, \dots, x_i\}$ assume that  $x_0<\dots< x_i$ .
    Then one defines:
    \begin{equation}
        gI_s=sign(g\restriction_s)I_{gs}
        \label{eq:BaseAction}
    \end{equation}
    First one observes that this is an action:
    \begin{proposition}
        $Aut(X)$ acts on $C^i(X)$ via extending the Map \ref{eq:BaseAction} linearly from the indicators to all of $C^i(X)$
    \end{proposition}
    \begin{proof}
        Since the actions is constructed by linear extension, one only needs to show that it is an action on the indicators of simplices. Namely, that for $g,h\in Aut(X)$ one has $(gh)I_s=sign(gh\restriction_s)I_{ghs}=sign(h\restriction_s)gI_{hs}$. To see this define three permutations: $\sigma_h, \sigma_g, \sigma_{gh}\in S_i$, where $\sigma_h$ is the permutation ordering $hx_0,\dots,hx_i$, $\sigma_{gh}$ the permutation ordering $ghx_0,\dots,ghx_i$ and $\sigma_g$ be the permutation ordering $g(\sigma_h(hx_0,\dots,hx_i))$. By definition one has:
        \begin{equation}
            hI_s=sign(\sigma_h)I_{hs}
        \end{equation}
        \begin{equation}
            ghI_s=sign(\sigma_{gh})I_{ghs}
        \end{equation}
        \begin{equation}
            gI_{hs}=sign(\sigma_g)I_{ghs}
        \end{equation}
        Tracing the definitions one notes that $\sigma_{gh}=\sigma_g\sigma_h$ and thus $sign(\sigma_{gh})=sign(\sigma_h)sign(\sigma_g)$. This proves:
        \begin{equation}
            g(hI_s)=sign(\sigma_h)gI_{hs}=sign(\sigma_h)sign(\sigma_g)I_{ghs}=sign(\sigma_{gh}) I_{ghs}=ghI_s
        \end{equation}
    \end{proof}
    With this technical lemma in hand one can prove invariance of $\Delta_i^+$ under automorphisms:
    \begin{proposition}
        For $g\in Aut(X)$ holds 
        \begin{equation}
            g^{-1}\Delta_i^+g=\Delta_{i}^+
        \end{equation}
        i.e. the signed up-down walk is $Aut(X)$-invariant. 
    \end{proposition}
    \begin{proof}
        Let $\{x_0,\dots,x_{i+1}\}=s\in X^{i+1}$ and $\{x_0,\dots, \hat{x}_k,\dots, x_{i+1}\}=s_k\in X^i$ for $0\leq k\leq i+1$. Then $\Delta_i^+(s_k,s_l)=(-1)^{k+l}$. For $g\in Aut(X)$ let $\sigma_s$ be the permutation induced by $g$ on the order of $\{x_0,\dots,x_{i+1}\}$ i.e. $gx_{\sigma(0)}<\dots< gx_{\sigma(i+1)}$ is ordered and $\sigma_{s_k}$ be the permutation induced by $g$ on the order of $\{x_0,\dots, \hat{x}_k,\dots, x_{i+1}\}$. Now one has
        \begin{dmath}g^{-1}\Delta_i^+g(v_k,v_l)=sign(\sigma_{s_k})sign(\sigma_{s_l})\Delta_i^+(gv_k,gv_l)=sign(\sigma_{s_k})(-1)^{\sigma(k)}sign(\sigma_s)sign(\sigma_{s_l})(-1)^{\sigma(l)}sign(\sigma_s)
        \end{dmath}
        \begin{claim}
            For $0\leq k\leq i+1$ one has 
            \begin{equation}
                sign(\sigma_{s_k})(-1)^{\sigma(k)}sign(\sigma_s)=(-1)^k 
            \end{equation}
        \end{claim}
        \begin{proof}
            Let $([j])$ be the cyclic permutation represented by the cycle $(0,\dots, j)$. The claim follows from the following identity:
            \begin{equation}
                ([\sigma(k)])^{-1}([i+1])\sigma_{s_k}([i+1])^{-1}([k])=\sigma_v
            \end{equation}
            The identity holds because $([k])$ swaps $gx_k$ to the front while the tail becomes $gs_k$. Then $([i+1])\sigma_{s_k}([i+1])^{-1}$ orders the tail and $([\sigma(k)])^{-1}$ reinserts $gx_k$ at the position $\sigma(k)$ i.e. where it is supposed to sit when ordering $gs$. Now the sign-function is a homomorphism to $\{\pm1\}$ and thus the identity implies the identity of the signs above. 
        \end{proof}
    \end{proof}

    \subsection{Weighted Graphs and Their Quotients}
    Instead of thinking of the signed up-down operator as defined via boundary and co-boundary map, one may take a slightly different perspective via \textbf{weighted graphs}:
    \begin{definition}
		A \textbf{weighted} graph $X$ is a triple $X=(V, E, \mu)$, where $V$ denotes the finite vertex set, $E\subset V\times V$ set of directed edges and $\mu$ is a weight function i.e. $\mu:E\rightarrow \mathbb{R}$. The \textbf{adjacency operator} of a weighted graph is defined as:
        \begin{equation}
            A^X(x,y)=\mu(x,y)
        \end{equation}
        The spectrum of the adjacency operator is called the \textbf{spectrum} of the weighted graph.
	\end{definition} 
    Let $X^i$ be the weighted graph which has $X_{n-2,q}(i)$ as vertex set and two vertices $x,y$ are connected by an edge if $|x\cap y|\in \{i,i+1\}$ i.e. the two simplices $x$ and $y$ have $i$-vertices in common or $x=y$. As a weight function on the edges one may assign to the edge $(x,y)\in X_i$ the weight $\mu(x,y)=\Delta_i^+(x,y)$. One notes that $\Delta_i^+(x,y)=0$ if $|x\cap y|\notin\{i,i+1\}$ and thus almost tautologically $\Delta_i^+$ becomes the adjacency operator of the weighted graph $X^i$.\\
    As shown before the operator $\Delta_i^+$ is $G$-invariant which implies when tracing the definitions that the $G$ acts on the graph $X^i$. $g\in G$ maps $x\in X^i$ to $gx$ while preserving adjacency \textbf{and} the weight function i.e. for all $x,y\in X^i$ the weights $\mu(x,y)=\Delta_i^+(x,y)=\Delta_i^+(gx,gy)=\mu(gx,gy)$. Formulated differently $G$ is a group of automorphisms of the weighted graph $X^i$. Now a group of automorphisms of a weighted graph can be used to define a \textbf{quotient} of the weighted graph:
    \begin{definition}
		Let $X=(V,E,\mu)$ be a weighted graph and $Aut(X)$ its group of automorphims. For $G\leq Aut(X)$ we define the quotient $X/G$ as follows: Let $V'\equiv V/G$ i.e. the vertex set $V'$ of the quotient is the space of orbits of $G$ acting on the vertex set $V$ of $X$. There is an edge between two vertices $y_1,y_2\in V'$ iff there are vertices $x_1,x_2\in V$ with $x_i\in y_i$ (keep in mind $y_i$ is a $G$-orbit of vertices) such that $(x_1, x_2)\in E$. For the weight function fix $x\in y_2$ arbitrarily, then $\mu'((y_1, y_2))$ for $(y_1,y_2)\in E'$ is given by:
		\begin{equation}
			\mu'((y_1,y_2))=\sum_{z\sim x  \& z\in y_1}\mu((x,z))
		\end{equation} 
		\label{def:QuotientWeights}
	\end{definition}
    Firstly, one notes that the choice of $x$ in the definition is immaterial. Namely if,  $x=gx'$ and $x\sim z\in y_2$ then $x'\sim g^{-1}z\in y_2$. The invariance of the weight  function $\mu$ under the action of $Aut(X)$ implies that the sum used to define $\mu'$ is independent of the choice of $x\in y_1$.\\
    In general we will denote the quotient map by $\pi:V\rightarrow V/G$. The quotient map induces a map $\pi^*$ from the functions on the quotient $V/G$ to the functions on the weighted graph $X$ by setting for $f\in \mathbb{R}^{V/G}$:
    \begin{equation}
        \pi^*f(x)=f(\pi(x))
        \label{eq:DualMap}
    \end{equation}
	Quotients of weighted graphs are very useful in our context as they allow for determining the spectrum of $X$ as the next few propositions will show. 
	\begin{proposition}
		Let $X=(V, E, \mu)$ be a weighted graph as above and $G\leq Aut(X)$ a group of automorphisms of $X$. Let $\tilde{X}=X/G$ defined as above. Then the spectrum of $\tilde{X}$ is contained in the spectrum of $X$ i.e. $spec(A^{\tilde{X}})\subset spec(A^X)$. 
		\label{prop:SpectralContainment}
	\end{proposition}
	\begin{proof}
		Let $f\in \mathbb{R}^{V/G}$ be an eigenfunction of $A^{\tilde{X}}$ with corresponding eigenvalue $\lambda$. Now one can can pull back $f$ via the quotient map $\pi$ to the function $\pi^*f$ on $X$. For $\pi^*f$ the following equalities hold:
		\begin{dmath}
			A^X \pi^*f(x)=\sum_{y\sim x}\mu((y,x)) \pi^*f(y)=\sum_{\pi(y)\sim\pi(x)}\sum_{z\in \pi(y), z\sim x}\mu((z,x))\pi^*f(z)=\sum_{\pi(y)\sim\pi(x)}\left(\sum_{z\in \pi(y), z\sim x}\mu((z,x))\right)f(\pi(y))=\sum_{\pi(y)\sim\pi(x)} A^{\tilde{X}}f(\pi(y))=A^{\tilde{X}}f(\pi(x))=\lambda f(\pi(x))=\lambda \pi^*f(x)
			\label{eq:EigenArithmetic}
		\end{dmath}
        Where we abused notation by denoting $z$ and $y$ being in the same $G$-orbit by $z\in \pi(y)$. Here the non-trivial step is the second equality splitting the summation over neighbors of $x$ into summing first over orbits and then summing over neighbors of $x$ contained in the orbits. Clearly, the above equality implies the proposition.
	\end{proof}
    The proof of the proposition shows that $\pi^*$ sends eigenfunction to eigenfunction of the relevant adjacency operators while preserving the associated eigenvalue. 
	Actually the proof shows the stronger statement that the spectrum of $X/G$ is exactly the $G$-invariant part of the spectrum of $X$. If there is an eigenfunction $f$ of $A^X$ with eigenvalue $\lambda$ which is also $G$-invariant then $\lambda$ is contained in the spectrum of $X/G$:
	\begin{corollary}
		Let $f$ be an eigenfunction of $A^X$ with corresponding eigenvalue $\lambda$, $G\leq Aut(X)$ and let $\tilde{X}=X/G$ . If $\sum_{g\in G} f^g\neq 0$ then $\lambda\in spec(A^{\tilde{X}})$.
		\label{cor:PreciseSpectrum}
	\end{corollary}
	\begin{proof}
		For an eigenfunction $f$ of $\Delta_i^+$  with eigenvalue $\lambda$ and $g\in G$ one has by invariance of $\Delta_i^+$:
        \begin{equation}
            \Delta_i^+gf=g\Delta_i^+f=  g\lambda f=\lambda gf  
        \end{equation}
        meaning that for all $g\in G$ $gf$ is an eigenfunction to eigenvalue $\lambda$ and so is their sum over $G$. Now set $f^G(u)=\frac{1}{|G|}\sum_{g\in G}f^g(u)$ i.e. the average of $f$ over $G$ and note that $f^G$ is constant on $G$-orbits. This allows for defining a function $f'$ on the quotient by setting $f'(\pi(x))=f^G(x)$. Now from Equation \ref{eq:EigenArithmetic} it follows:
		\begin{equation}
			A^{\tilde{X}}f'(\pi(x))=A^Xf^G(x)=\lambda f^G(x)=\lambda f'(\pi(x))
		\end{equation}
		This shows that if $f^G$ and thus $f'$ is not zero then $\lambda\in Spec(A^{\tilde{X}})$.
	\end{proof}
    Before turning to actually proving Papikian's conjecture, we remark that passing to the quotient $X/G$ is equivalent to restricting the adjacency $A^X$ to functions on $X$ invariant under $G$. The upside of the quotient approach is that it makes very explicit what the restricted adjacency looks like. 

    \subsubsection{The Edge Weights of $X^i$}
    \label{sec:EdgeWeights}
    Before discussing the quotients of $X^i$ it is worth to determine the weight of the edges in $X^i$. To this end let $t,u$ be two adjacent vertices of $X^i$ i.e. $U,W$ are two $i$-simplices in $X_{n,q}$ such that $s=t\cup u$ is an $i+1$ simplex of $X_{n,q}$. As each simplex of $X_{n,q}$ corresponds to a partial flag of $\mathbb{F}_q^n$ we may identify $s$ with the flag $\mathbf{x}_0\subset\dots \mathbf{x}_i$ of appropriate subspaces of $\mathbb{F}_q^n$. Further, as $W,U$ are contained in $s$ there is $0\leq k,l\leq i+1$ such that $t=\hat{s}_k=(\mathbf{x}_0\subset\dots\subset \mathbf{x}_{k-1}\subset \mathbf{x}_{k+1}\subset\dots \subset \mathbf{x}_{i+1})$ (and $u=\hat{s}_l$ respectively) i.e. $t$ is the $s$  with the $k$-th vertex according to color removed. Plugging in the definition of the signed up-down walk i.e. Equations \ref{eq:CoBoundaryOperator} for the coboundary operator and Equation \ref{eq:BoundaryOperator} for the boundary operator one has:
    \begin{equation}
        sign(\Delta_i^+(\hat{s}_k,\hat{s}_l))=(-1)^{k+l}
    \end{equation}
    To determine $|\Delta_i^+(\hat{s}_k, \hat{s}_l)|$ one needs to calculate $\mu(s)$ and $\mu(\hat{s}_k)$ according to Equation \ref{eq:BoundaryOperator}. 
    \begin{lemma}
        Let $s$ be an $i$-chain of type $\{c_0,\dots, c_i\}$. Setting $c_{-1}=0$ and $c_{i+1}=n$ one has:
        \begin{equation}
            \mu(s)=\prod_{j=0}^{i+1}\prod_{m=0}^{c_j-c_{j-1}}\gBinom{c_j-c_{j-1}-m}{1}_q
            \label{eq:SimplexWeight}
        \end{equation}
        where $\gBinom{a}{b}_q$ denotes the Gaussian binomial coefficient as defined in Section \ref{sec:CountingSubspaces} of the appendix. 
    \end{lemma}
    \begin{proof}
        Every top dimensional simplex containing $V$ corresponds to a unique full flag of $\mathbb{F}_q^n$ containing the partial flag $\mathbf{x}_0\subset\dots\subset \mathbf{x}_{i}$ and vice versa. So in order to count top-level simplices containing $s$ one needs determine how many options there are to complete $\mathbf{x}_0\subset\dots\subset \mathbf{x}_{i}$ to a full flag. Completing the partial flag $\mathbf{x}_0\subset\dots \subset \mathbf{x}_i$ between $\mathbf{x}_j$ and $\mathbf{x}_{j+1}$ is independent of the completion between $\mathbf{x}_{j'}$ and $\mathbf{x}_{j'+1}$ for $j\neq j'$. Further, completing the partial flag between $\mathbf{x}_j$ and $\mathbf{x}_{j+1}$ is equivalent to choosing a full flag for $\mathbf{x}_{j+1}/\mathbf{x}_j$. The space $\mathbf{x}_{j+1}/\mathbf{x}_j$ has dimension $c_{j+1}-c_j$ by the definition of the type. Then Proposition \ref{prop:FullFlags} asserts that $\prod_{m=0}^{c_{j+1}-c_j}\gBinom{c_{j+1}-c_j-m}{1}_q$ options to complete the partial flag between $\mathbf{x}_j$ and $\mathbf{x}_{j+1}$. Taking the product over $0\leq j\leq i+1$ one obtains Formula \ref{eq:SimplexWeight} for the weight of $s$. 
    \end{proof}
    With this formula in hand determining $\frac{\mu(s)}{\mu(\hat{s}_k)}$ becomes a mere calculation:
    \begin{corollary}
        Let $s=\{\mathbf{x}_0,\dots, \mathbf{x}_i\}$ with profile $c_0<\dots<c_i$ and let $\hat{s}_m$ be the simplex $s$ with vertex $k$ removed i.e. $\hat{s}_k=\{\mathbf{x}_0,\dots,\mathbf{x}_{k-1},\mathbf{x}_{k+1},\dots, \mathbf{x}_i\}$. Then one has
        \begin{equation}
            \frac{w(s)}{s(\hat{s}_k)}=\frac{1}{\gBinom{c_{k+1}-c_{k-1}}{c_k-c_{k-1}}_q}
        \end{equation}
        again with the convention that $c_{-1}=0$ and $c_{i+1}=n$
        \label{cor:EdgeWeight}
    \end{corollary}
    \begin{proof}
        According to Equation \ref{eq:SimplexWeight} one has:
        \begin{dmath}
             \frac{w(s)}{s(\hat{s}_k)}=\frac{\prod_{j=0}^{i+1}\prod_{m=0}^{c_j-c_{j-1}}\gBinom{c_j-c_{j-1}-m}{1}_q}{\prod_{j=0}^{k-1}\prod_{m=0}^{c_j-c_{j-1}}\gBinom{c_j-c_{j-1}-m}{1}_q \prod_{j=k+1}^{i+1}\prod_{m=0}^{c_j-c_{j-1}}\gBinom{c_j-c_{j-1}-m}{1}_q \prod_{m=0}^{c_{k+1}-c_{k-1}}\gBinom{c_{k+1}-c_{k-1}-m}{1}_q}=
             \frac{\prod_{m=0}^{c_{k}-c_{k-1}}\gBinom{c_{k}-c_{k-1}-m}{1}_q \prod_{m=0}^{c_{k+1}-c_{k}}\gBinom{c_{k+1}-c_{k}-m}{1}_q}{\prod_{m=0}^{c_{k+1}-c_{k-1}}\gBinom{c_{k+1}-c_{k-1}-m}{1}_q}=\frac{1}{\gBinom{c_{k+1}-c_{k-1}}{c_k-c_{k-1}}_q}
        \end{dmath}
    \end{proof}

    \section{The Proof}
    Before diving into the details of the proof, an outline of the proof's strategy is in place. The first part of the theorem, i.e. the uniform bound on the number of distinct eigenvalues in terms of $n$ and $i$, is due the smart choice of a quotient. Namely, let $H$ be the stabilizer of a top level simplex in the spherical building. Then Proposition \ref{prop:SpectralEquivalence} shows that the quotient $X^i/H$ has the same spectrum as a set as $X^i$ itself. After determining the structure of the quotient map Proposition \ref{prop:QuotientStructure} states that the vertex set of $X^i/H$ can be embedded into the space of $i$-flags of the interval $[n]$. This embedding then is turned in Proposition \ref{thm:NumberEigenvalues} into a bound on the number of distinct eigenvalues of $X^i$ in terms of $n$ and $i$ but crucially independent of $q$.\\
    Determining the limits of the eigenvalues is split into three steps: First one needs to establish the limits of each of the entries of the adjacency matrix of the quotient $X^i/H$ as is done in Section \ref{sec:TheLimit}. One of the main take-aways from section \ref{sec:TheLimit} the limit is that the symmetric version of this limit is a very sparse matrix. In Section \ref{sec:TheBlock}, while looking at the vanishing condition for each of the entries Proposition \ref{prop:BlockIntervals} establishes a block decomposition of this symmetric matrix. This block decomposition looks on its face somewhat counter-intuitive since blocks usually correspond to connected components in a graph but the spherical building is a connected expander. To resolve this, one notes that the walk one considers is not the original random walk but a symmetrized version of it, meaning that it takes into account the size of the $H$-orbit corresponding to each vertex. Since the orbits have vastly different size the transition probability between many orbits must vanish. Proposition \ref{prop:BlockSimplex} finally establishes that each block is equivalent to the up-down walk on a complete complex. In Section \ref{sec:SimplexSpecs} the spectrum of the $i$-walk on an $n$-dimensional complete complex is established. Finally, combining these spectra with the decomposition of the symmetrized walk matrix of the limit from section \ref{sec:TheBlock} one obtains the second part of the main theorem in Proposition \ref{thm:Eigenstripping}.
    \subsection{The Quotient}
    \label{sec:TheQuotient}
    In order to prove Papikian's conjecture, the most important step is to choose an appropriate quotient.
    \begin{proposition}
        Let $X^i$ be the weighted graph induced by $\Delta_i^+$ as described above. Let $H$ be the stabilizer of a top-level cell of $X_{n,q}$ and let $\tilde{X}^i=X^i/H$ denote the quotient. Then as a set i.e. discounting multiplicity the spectrum of the $X^i$ and $\tilde{X}^i$ are the same. 
        \label{prop:SpectralEquivalence}
    \end{proposition}
    \begin{proof}
        Let $s$ denote the top-level simplex stabilized by $H$. For $f\neq 0$ an eigenfunction of $X^i$ with eigenvalue $\lambda$ there is $t'\in X_{n,q}(i)$ such that $f(t')\neq 0$. Since $G$ acts transitively on each type of cells in $X_{n-2,q}$ and there is a subsimplex $t\subset s$ such that $type(t)=type(t')$ according to Lemma \ref{lem:MasterTransitivity} there is $g\in G$ which maps $t'$ to $t$ i.e. $gt'=t$. Therefore $gf(t)=f(g^{-1}t)=f(t')\neq 0$. Further, $H$ preserves the type of each simplex and as it stabilizes $s$, it must also stabilize $t$ since $t$ is the unique simplex of type $type(t)$ contained in $s$. This means that $\sum_{h\in H}hgf(t)=|H|gf(t)\neq 0$ and by Corollary \ref{cor:PreciseSpectrum} this implies that $\lambda \in spec(\tilde{X}^i)$. Since $\lambda$ was arbitrary this proves the $spec(X^i)=spec(X^i/H)$. 
    \end{proof}
    Observe that $G$ acts transitively on all top-level cells of $X_{n,q}$ which means that the stabilizers of the top-level cells are all $G$-conjugates of each other. This in particular implies that for two top level cells $s$ and $s'$ with stabilizer $H$ and $H'$ respectively, the two quotients $X^i/H$ and $X^i/H'$ are isomorphic. Therefore, we may choose the top-level cell that makes calculations the easiest.\\
    In order to choose our favorite top-level cell, fix a basis $\{e_1,\dots, e_n\}$ of $\mathbb{F}_q^n$. With respect to this basis let $x_i$ be the vertex of $X_{n-2,q}$ corresponding to $span(\{e_1,\dots, e_i\})$ and let $s=\{x_1,\dots,x_{n-1}\}$ (note, that $x_n$ is not a vertex of $X_{n-2,q}$ as it corresponds to $\mathbb{F}_n^q$ i.e. a trivial subspace). 
    \begin{lemma}
        The stabilizer of $s$ with respect to the basis $\{e_1,\dots, e_n\}$ is the group of upper triangular matrices with non-zero entries on the diagonal. 
    \end{lemma}
    \begin{proof}
        Let $h\in H$. Then one has $h \cdot span(\{e_1,\dots, e_i\}=span(\{e_1,\dots, e_i\})$. This in particular implies that $h\cdot e_i\in span(\{e_1,\dots, e_i\}$. Since this holds for all $i$ the matrix $h$ is upper-triangular. 
    \end{proof}
    In order to determine the orbits of $H$ in $X^i$ and thus the vertices of $X^i/H$ some definitions are needed:
    \begin{definition}
        Let $v=\sum a_ie_i$ be a vector in $\mathbb{F}_q^n$ then we say its \textbf{height} denoted by $height(v)$ is the largest $i$ such that $a_i\neq 0$. Further, for a collection of vectors $V=\{v_1,\dots, v_i\}$ one refers to the set $profile(v)=\{height(v_1),\dots, height(v_i)\}$ as the \textbf{height profile}
        \label{def:Heights}
    \end{definition}
    The following technical propositions will show that the height-profiles index the $H$-orbits of $X^i$ i.e. it is equivalent to the quotient map $\pi$. The first indication of this fact is: 
    \begin{proposition}
        Let $v_1,\dots v_i$ be a set of vectors of pair-wise different height. Then there is $h\in H$ such that $hv_j=e_{height(v_j)}$.
        \label{prop:SortHeights}
    \end{proposition}
    \begin{proof}
        Set $e_1',\dots, e_n'$ such that $e_k'=v_j$ if $k=height(v_j)$ and otherwise $e_k'=e_k$. Interpreting the vectors as column vectors the matrix $h'=(e_1'\dots e_n')$ is upper triangular and non-degenerate since each diagonal entry is non-zero. Since the set of non-degenerate upper-triangular matrices forms a group there is an upper-triangular matrix $h\in H$ such that $hh'=I_n$ i.e. $h=(h')^{-1}$. Reading $hh'=I_n$ column by column one has for each vector $v_j$ that $hv_j=e_{height(v_j)}$ as required. 
    \end{proof}
    Next one needs to extend the notion of height-profile to a subspace of $\mathbb{F}_q^n$. Let $\textbf{x}\subset \mathbb{F}_q^n$ be a subspace of dimension $i$ and let $\{v_1,\dots, v_i\}$ be a basis. For such a basis define the following algorithm:
    \begin{enumerate}
        \item normalize each basis element such that the $e_{height(v_k)}$-coefficient of $v_k$ is $1$
        \item if two vectors $v_{k_1}, v_{k_2}$ have the same height replace $v_{k_2}$ by $v_{k_2}-v_{k_1}$
        \item if there remain two vectors of the same height start the loop again
    \end{enumerate}
    For a basis after being treated by this algorithm let $\{height(v_1),\dots,height(v_i)\}$ be its height profile
    \begin{proposition}
        Let $\textbf{x}\subset \mathbb{F}_q^n$ be a subspace of dimension $i$, then its height profile is independent of the basis chosen, i.e. the height profile of any basis is the same after the above algorithm is applied. 
        \label{prop:UniqueHeights}
    \end{proposition}
    \begin{proof}
        Let $\{v_1, \dots,v_i\}$ and $\{v_1'\dots, v_i'\}$ be two bases after the normalization process such that $height(v_j)<height(v_{j+1})$. Observe that $height(v_i)=height(v_i')$ as otherwise $v_i\notin span(\{v_1'\dots, v_i'\})$ or $v_i'\notin span(\{v_1, \dots,v_i\})$. This implies that $span(\{v_1,\dots, v_{i-1}\})=\mathbf{x}\cap span(\{e_1,\dots,e_{height(v_i)-1}\})=span(\{v_1',\dots, v_{i-1}'\})$. By induction the statement holds. 
    \end{proof}
    The preceding proposition shows that the \textbf{height profile map} from $X(0)$ to the power-set of $[n]$ is well defined.  Combining the two propositions allows for determining the vertices of $X^i/H$
    \begin{proposition}
        Let $[n]=\{1,\dots,n\}$ then the vertices of $X^i/H$ correspond via the height profile map to the $i$-flags of subsets of $[n]$ ordered by inclusion. 
        \label{prop:QuotientStructure}
    \end{proposition}
    \begin{proof}
        For $i=0$ the last two propositions prove the statement: Proposition \ref{prop:SortHeights} shows $H$ acts transitively on the set of subspaces $\mathbf{x}$ with the same height profile as they can be mapped to $span(\{e_j|j\in profile(\mathbf{x})\}$. Since any matrix in $H$ is invertible and upper-triangular, it preserves the height of any vector and by extension it preserves the height-profile of any subspace of $\mathbb{F}_q^n$. These two observations together show the set of subspaces with the same height profile is one $H$-orbit. Since for any subset $V\subset [n]$ the vector space $span\{e_j|j\in V\}$ has height profile $V$, the $H$-orbits of $X^0$ are in bijection with the proper subsets of $[n]$ via the height profile map.\\
        In order to extend the result to $i>0$ note that for the subspaces $\mathbf{x}_0\subset\dots\subset \mathbf{x}_i$ one can find a sequence of sets of vectors $V_0\subset\dots\subset V_i$ such that $V_j$ is a basis of $\mathbf{x}_j$. In this case one can modify the algorithm above such that whenever there are two vectors $v_{j_1}\in V_{j_1}$ and $v_{j_2}\in V_{j_2}\backslash V_{j_1}$ then one replaces $v_{j_2}$ by $v_{j_2}-v_{j_1}$. That way one ensures that the normalization algorithm produces a sequence of normalized bases $V_0'\subset\dots\subset V_i'$ such that $span(V_j)=span(V_j')$. Now applying $h\in H$ which sends $V_i'$ to $\{e_j|j\in profile(V_i')\}$ sends $V_k'$ to $\{e_j|j\in profile(V_k')\}$ i.e. sends $\mathbf{x}_j$ to $span(\{e_j|j\in profile(V_k')\})$. This shows that two flags of subspaces with the same height profile for the corresponding flag members are in the same $H$-orbit. As in the case $i=0$ the group of upper triangular matrices $H$ preserves the height-profile, which means that the set of flags with a fixed height profile form an $H$-orbit in $X^i$. Further, given any flag $V=V_0\subset \dots \subset V_i$ of $[n]$ the flag $span\{e_j|j\in V_0\}\subset \dots\subset span\{e_j|j\in V_i\}$ of $\mathbb{F}_q^n$ has height profile $V$. Therefore, $H$-orbits of $X^i$ and $i$-flags of $[n]$ are in bijection via the height profile map.   
    \end{proof}
    From now on we will refer to the height profile map by $\pi$. The preceding shows that $\pi: X^i\rightarrow X^i/H$ acts as the quotient map. As a corollary one obtains the first part of Papikian's conjecture:
    \begin{proposition}
        The spectrum of the $i$-th signed up-down walk $\Delta_i^+$ has at most 
        \begin{equation}
            \sum_{k=0}^{i+1}\binom{i+2}{k} (-1)^k (i+2-k)^n
            \label{eq:QuotientSize}
        \end{equation}
        pairwise different eigenvalues.
        \label{thm:NumberEigenvalues}
    \end{proposition}
    \begin{proof}
        Proposition \ref{prop:SpectralEquivalence} shows that the spectrum of $X^i$ and $X^i/H$ are equivalent up to multiplicity which means $|spec(X^i)|=|spec(X^i/H)|$. In order to determine $|spec(X^i/H)|$ we may use the trivial bound $|spec(X^i/H)|\leq |X^i/H|$ since a graph cannot have more eigenvalues that its size. But as shown in Proposition \ref{prop:QuotientStructure} $|X^i/H|$ depends only on $n$ and $i$ but not on $q$ so independence of $q$ is already guaranteed.\\
        To derive Formula \ref{eq:QuotientSize} for $|X^i/H|$ consider the following: For a flag $V_{-1}=\emptyset\subset V_0\subset\dots\subset V_i\subset V_{i+1}=[n]$ we may equivalently consider the sequence of sets $V_j'=V_{j}\backslash V_{j-1}$  for $0\leq j\leq i+1$. With this notation one has $V_j=\bigsqcup_{k\leq j} V_k'$. Since $V_j$ is a proper subset of $V_{j+1}$ no $V_j'$ is empty. Further, $V_i\neq [n]$ since the vertices $X$ correspond to non-trivial subspaces of $\mathbb{F}_q^n$ and the only space in $\mathbb{F}_q^n$ with height-profile  $[n]$ is the entire space.  This means that the vertices of $X^i/H$ are in bijection with the partitions of $[n]$ into $i+2$ ordered bins such that no bin is empty. These partitions can be counted as follows: There are $(i+2)^n$ partitions of $[n]$ into $i+2$-bins. From those one needs to discount the partitions which leave at least one bin empty. Fixing an empty bin there are $(i+2-1)^n$ partitions which leave this bin empty. So the number of partitions leaving one bin empty can be approximated by $\binom{i+2}{1}(i+2-1)^n$. This approximation counts partitions which leave more than one bin empty multiple times. To correct for that one adds again the approximate number of partitions that leave at least two bins empty $\binom{i+2}{2}i^n$. Continuing this counting method until $k=i+1$ i.e. at the number of assignments that leave exactly $i+1$ bins empty, which is $\binom{i+2}{i+1}$, one arrives at Expression \ref{eq:QuotientSize} for the size of $H^i/X$.
    \end{proof}
    \subsection{The Limit}
    \label{sec:TheLimit}
    To actually determine the asymptotics of the eigenvalues for $q\rightarrow\infty$ the strategy is to calculate the asymptotics of the adjacency operator $\tilde{\Delta}_i^+$ of $X^i/H$ for $q\rightarrow\infty$ and then calculate the eigenvalues. As we will see this is significantly easier than determining the eigenvalues and then determine their asymptotics since $\tilde{\Delta}_i^+$ loosely speaking converges to a very sparse matrix.\\
    To justify this approach we present a few statements:
    \begin{lemma}
        Let $A_j\in M_{k\times k}(\mathbb{R})$ be a sequence of real $k\times k$ matrices such that the sequence converges entry-wise to $A$. Then $spec(A_j)$ converges as a set to $spec(A)$.
        \label{lem:SpectralConvergence}
    \end{lemma}
    \begin{proof}
        Since the $A_j$ converge entry-wise to $A$ their characteristic polynomials $p_j$ must converge to $p$ the characteristic polynomial of $A$ meaning that the $m$-th coefficient of $p_j$ converges to the $m$-th coefficient of $p$. Now roots of a polynomials are continuous i.e. the roots of $p_j$ must converge to the roots of $p$ meaning that $spec(A_j)=roots(p_j)$ converges as a set to $spec(A)=roots(p)$.
    \end{proof}
    As defined right now $\tilde{\Delta}_i^+$ might not converge entry-wise. Since we are only interested in the spectrum of $\tilde{\Delta}_i^+$ we may equivalently calculate the spectrum of a symmetric conjugate of $\tilde{\Delta}_i^+$ whose entries converge.
    \begin{proposition}
        The operators $\tilde{\Delta}_i^+$ can be symmetrized via conjugation by a diagonal matrix depending on $n,i$ and $q$. 
        \label{prop:SymMatrix}
    \end{proposition}
    \begin{proof}
        For a vertex $k\in X^i/H$ be two vertices and denote by $f^k$ the indicator of the $H$-orbit in $X^i$ corresponding to $k$. For any two distinct vertices $k,l\in X^i/H$ the support of $f^k$ and $f^l$ is disjoint meaning they are orthogonal with respect to the scalar product induced by the weight function. This implies that in the vector-space spanned by $\{f^k|k\in X^i/H\}$ with the scalar product induced by the weight function the dual function of $f^k$ is $\frac{f^k}{<f^k,f^k>}$. This means for the adjacency of $X^i/H$:
        \begin{equation}
            \tilde{\Delta}_i^+(k,l)=\iprod{f^k\Delta_i^+, \frac{f^l}{\iprod{f^l,f^l}}}
        \end{equation}
        Conjugating by $\mu=diag(\{\sqrt{\iprod{f^k,f^k}}\})$ gives:
        \begin{equation}
            \mu^{-1}\tilde{\Delta}_i^+\mu(k,l)=\iprod{\frac{f^k}{\sqrt{\iprod{f^k,f^k}}}\Delta_i^+, \frac{f^l}{\sqrt{\iprod{f^l,f^l}}}}=\iprod{\frac{f^k}{\sqrt{\iprod{f^k,f^k}}}, \Delta_i^+\frac{f^l}{\sqrt{\iprod{f^l,f^l}}}}=\mu^{-1}\tilde{\Delta}_i^+\mu(l,k)
        \end{equation}
        i.e. shows that $\mu^{-1}\tilde{\Delta}_i^+\mu$ is symmetric. 
    \end{proof}
    Given that $\mu$ is diagonal one does not need to actually calculate $\mu$ since in this case for $\mu^{-1}\tilde{\Delta}_i^+\mu$
    \begin{equation}
        \mu^{-1}\tilde{\Delta}_i^+\mu(k,l)=sign(\Delta_i^+(k,l))\sqrt{\tilde{\Delta}_i^+(k,l)\tilde{\Delta}_i^+(l,k)}
        \label{eq:SymmetricCurvature}
    \end{equation}
    holds.\\
    Turning to determining the entries of $\tilde{\Delta}_i^+$ let $V=(V_0\subset\dots\subset V_{i+1})\in X^{i+1}/H$. For $0\leq k\leq i+1$  let $\hat{V}_k$ be the corresponding vertex in $X^i/H$. If $\hat{s}_k$ in $H$-orbit $\pi^{-1}(\hat{V}_k)$ and a neighbor $\hat{s}_l$ in $H$-orbit $\pi^{-1}(\hat{V}_l)$ then the edge between the two has weight $(-1)^{k+l}/\gBinom{c_{k+1}-c_{k-1}}{c_k-c_{k-1}}_q$ according to Corollary \ref{cor:EdgeWeight}. So in order to determine the weight of the edge between $\hat{V}_k$ and $\hat{V}_l$ in $X^i/H$ one needs to count the neighbors of $\hat{s}_l$ in $\pi^{-1}(\hat{V}_k)$.

    Counting neighbors we split in two steps according to $\Delta_i^+=\delta_id_i$ i.e. in order to transition from a simplex in $\pi^{-1}(\hat{V}_k)$ to a simplex in $\pi^{-1}(\hat{V}_l)$ one necessarily needs to pass through a simplex in $H$-orbit $\pi^{-1}(V)$. Each simplex $s\in \pi^{-1}(V)$ contains exactly one simplex in orbit $\pi^{-1}(\hat{V}_k)$. If two distinct simplices $s,s'\in \pi^{-1}(V)$ contain $\hat{s}_k$ then $\hat{s}_l\neq\hat{s}_l'$ for $l\neq k$. Therefore, the number of simplices in orbit $\pi^{-1}(V)$ containing $\hat{s}_l$ is the same as neighbors in orbit $\hat{V}_k$.
    \begin{proposition}
        Let $V$ be the vertex in $X^{i+1}/H$ with label $\emptyset=V_{-1}\subset V_0\subset\dots\subset V_i\subset V_{i+1}=[n]$ then a simplex $\hat{s}_k\in\pi^{-1}(\hat{V}_k)$ is contained in  
        \begin{equation}
            \prod_{j\in V_k\backslash V_{k-1}}q^{\{j'\in V_{k+1}\backslash V_k|j'<j\}}
            \label{eq:NoNeighbors}
        \end{equation}
        many simplices in $H$-orbit $\pi^{-1}(V)$. 
    \end{proposition}
    \begin{proof}
        Let $\hat{s}_k=\{x_0,\dots, x_{k-1},x_{k+1}, \dots, x_i\}\in \hat{V}_k$ then the number of simplices containing $t$ in $\pi^{-1}(V)$ is equal to the number of distinct subspaces of height profile $V_k$ sandwiched between $\mathbf{x}_{k-1}$ and $\mathbf{x}_{k+1}$.\\
        Since both $\pi^{-1}(V)$ and $\pi^{-1}(\hat{V}_k)$ are $H$-orbits we may assume w.l.o.g. that $t$ is our favorite simplex in $\pi^{-1}(\hat{V}_k)$ namely that $\mathbf{x}_l=span(\{e_j|j\in V_l\})$ for $l\neq k$.\\
        Before proving the proposition in general consider the example of $V_{k-1}=\{1,3\}$,$V_k=\{1,3,6\}$, $V_{k+1}=\{1,3,4,6,9\}$: A sandwiched space $\mathbf{x}_k$ of height profile $\{1,3,6\}$ corresponds to a one dimensional subspace in the quotient $\mathbf{x}_{k+1}/\mathbf{x}_{k-1}$. Choosing as a natural basis for $\mathbf{x}_{k+1}/\mathbf{x}_{k-1}$ the basis $e_1'=[e_4]$, $e_2'=[e_6]$ and $e_3'=[e_9]$ the line $[\mathbf{x}_k]$ must be spanned by a vector of the form $\alpha_1 e_1'+\alpha_2 e_2'$ with $\alpha_2\neq 0$ i.e. a vector of height 2.  In $\mathbb{F}_q^n$ there are $q$ lines spanned by a vector of height 2, since w.l.o.g. one may set $\alpha_2=1$ and choose $\alpha_1\in \mathbb{F}_q$ for which there are $q$-choices.\\
        In the general case one also passes to the quotient while carefully keeping track of the heights.  For $\mathbf{x}_{k-1}\subset \mathbf{x}_k\subset \mathbf{x}_{k+1}$ with respective height profiles $V_{k-1}\subset V_k\subset V_{k+1}$, enumerate the indices in $V_{k+1}\backslash V_{k-1}$ by $m_j$ in an ascending manner. Fix as basis for $\mathbf{x}_{k+1}/\mathbf{x}_{k-1}$ the set $\{e_j'=[e_{m_j}]|1\leq j\leq |V_{k+1}|-|V_{k-1}|\} $. Note that for $m_j\in V_k\backslash V_{k-1}$ a vector $v\in V_k$ of height $m_j$ its image in the quotient $[v]$ has height $j$. Thus, each sandwiched space $\mathbf{x}_k$ of height profile $V_k$  corresponds to a subspace $[\mathbf{x}_k]\subset \mathbf{x}_{k+1}/\mathbf{x}_{k-1}$ of height profile $\tilde{V}_k=\{j|m_j\in V_k\backslash V_{k-1}\}$ and vice versa. So in order to count the number of sandwiched spaces one may equivalently count the number of subspaces of $\mathbb{F}_q^{|V_{k+1}\backslash V_{k-1}|}$ of height profile $\tilde{V}_k$. 
        \begin{claim}
            In the vector space $\mathbb{F}_q^n$, and height profile $W$ there are
            \begin{equation}
                \prod_{j\in W}q^{\{j'\in [n]\backslash W| j'<j\}}
                \label{eq:NoBoringSpaces}
            \end{equation}
            many subspaces with height profile $V$. 
        \end{claim}
        \begin{proof}
            To span a vector space of height profile $W$ one needs to choose for each $j\in W$ a vector $v_j$ of height $j$. To avoid double counting spaces one imposes two conditions:
            \begin{enumerate}
                \item the $j$-th entry of $v_j$ is $1$
                \item for $j'\in V\backslash\{j\}$ the $j'$-th entry is $0$
            \end{enumerate}
            Note that there are $q^{\{j'\in [n]\backslash V| j'<j\}}$-many options to choose $v_j$. To show that for each $j$ the choices are independent one needs to show that each $\mathbf{x}\subset \mathbb{F}_q^n$ contains a suitable $v_j$ and it is unique. Existence can be seen as follows: Choose basis vectors $v_j'\in \mathbf{x}$ such that the height of $v_j'$ is $j$. To produce $v_j$ for $j'$ use $v_{j'}'$ to set the $j'$-th entry zero for descending $j'$. Further, $v_j$ must be unique since the difference of two vectors of same height satisfying the above conditions is supported on entries outside $W$ i.e. cannot be in $\mathbf{x}$. Altogether, this shows that each different choice for the $v_j$ leads to a different space $\mathbf{x}$ and each $\mathbf{x}$ is associated to a basis of vectors satisfying the above conditions. Since one has $q^{\{j'\in [n]\backslash V| j'<j\}}$-options for $v_j$, the number of spaces $\mathbf{x}$ is given as in Expression \ref{eq:NoBoringSpaces}.
        \end{proof}
        Using the way we set up the height function in the quotient, the claim shows, when tracing the indices appropriately, that the number of sandwiched spaces of height profile $V_k$ is given as in Expression \ref{eq:NoNeighbors}.
    \end{proof}
    As a corollary one obtains a formula for the number of neighbors a simplex $s\in X^i$ has in a certain $H$-orbit $\pi^{-1}(\hat{V}_l)$ for $\hat{V}_l\in X^i/H$:
    \begin{corollary}
        Let $V$ be the height profile $\emptyset \neq V_0\subset \dots \subset V_{i+1}\neq [n]$. Let $k\neq l$ and let $\hat{V}_k$ and $\hat{V}_l$ be the height profile of $V$ removing $V_k$ and $V_l$ respectively. If $s\in X^i$ is in $H$-orbit $\pi^{-1}(V_k)$ then $s$ has
         \begin{equation}
            \prod_{j\in V_k\backslash V_{k-1}}q^{\{j'\in V_{k+1}\backslash V_k|j'<j\}}
            \label{eq:NoNeighbors2}
        \end{equation}
        many neighbors in $H$-orbit $\pi^{-1}(\hat{V}_l)$.
    \end{corollary}
    This last corollary allows for determining the edge weights of $X^i/H$ up to a sign:
    \begin{proposition}
            As before let $V$ be the height profile $\emptyset=V_{-1}\subset V_0\subset\dots\subset V_i\subset V_{i+1}=[n]$  and for $k\neq l$ let $\hat{V}_k$ and $\hat{V}_l$ be $V$ with $V_k$ and $V_l$ respectively removed. Then 
            \begin{equation}
                |\tilde{\Delta}_i^+(\hat{V}_k,\hat{V}_l)|=\frac{\prod_{j\in V_k\backslash V_{k-1}}q^{\{j'\in V_{k+1}\backslash V_k|j'<j\}}}{\gBinom{|V_{k+1}\backslash V_{k-1}|}{|V_{k}\backslash V_{k-1}|}_q}
                \label{eq:AbsoluteOffDiagonal}
            \end{equation}
            and on the diagonal one has 
            \begin{equation}
                \tilde{\Delta}_i^+(V_k,V_k)=n-i
                \label{eq:DiagonalEntries}
            \end{equation}
    \end{proposition}
    \label{prop:AbsoluteEntries}
    \begin{proof}
        The sign of $\Delta_i^+(s,t)$ for $s\in\pi^{-1} (V_k)$ and $t\in\pi^{-1}(V_l)$ only depends on $k$ and $l$. And $\Delta_i^+(s,t)=\pm \frac{1}{\gBinom{|V_{k+1}\backslash V_{k-1}|}{|V_{k}\backslash V_{k-1}|}_q}$ which explains Equation \ref{eq:AbsoluteOffDiagonal}.\\
        For diagonal elements consider the following: for $s=\{v_0,\dots,v_i\}$ with the vertices sorted according to color one has $sign(\{v_k\}\cup\{v_0,\dots,\hat{v}_k,\dots, v_i\})=(-1)^k$. This in particular implies that the signs for each loop coming from the boundary map and the co-boundary map cancel i.e. all loops contribute positively to the corresponding loop in $X^i/H$. Then the following claim proves Equation \ref{eq:DiagonalEntries}:
        \begin{claim}
            Let $s\in X^i$ be any $i$-dimensional simplex. Then holds:
            \begin{equation}
                \sum_{s\subset t\in  X(i+1)}w(t)=(n-i)w(s)
            \end{equation}
        \end{claim}
        \begin{proof}
            Since $w(s)=|\{u\in X(n)|s\subset n\}|$ and $X$ is a pure complex one has
            \begin{equation}
                \sum_{s\subset t\in  X(i+1)}w(t)=\sum_{s\subset u\in X(n)}|\{t\in X(i+1)|s\subset t\subset u\}|=\sum_{s\subset u\in X(n)}(n-i)=(n-i)w(s)
            \end{equation}
        \end{proof}
    \end{proof}
    Note that for a $i+1$-flag $V$ the  sign of $sign(\tilde{\Delta}_i^+(\hat{V}_k,\hat{V}_l))=sign(\tilde{\Delta}_i^+(\hat{V}_l,\hat{V}_k))=(-1)^{k+l}$.
    \begin{corollary}
        Let $\mu$ be as in Equation \ref{eq:SymmetricCurvature} and height profiles $V_k$ and $V_l$ as in Proposition \ref{prop:AbsoluteEntries}. Then one has:
        \begin{equation}
            \left(\mu^{-1}\tilde{\Delta}_i^+\mu\right)(\hat{V}_k,\hat{V}_l)=(-1)^{k+l}\sqrt{\frac{\prod_{j\in V_k\backslash V_{k-1}}q^{\{j'\in V_{k+1}\backslash V_k|j'<j\}}}{\gBinom{|V_{k+1}\backslash V_{k-1}|}{|V_{k}\backslash V_{k-1}|}_q} \frac{\prod_{j\in V_l\backslash V_{l-1}}q^{\{j'\in V_{l+1}\backslash V_l|j'<j\}}}{\gBinom{|V_{l+1}\backslash V_{l-1}|}{|V_{l}\backslash V_{l-1}|}_q}}
        \end{equation}
        And
        \begin{equation}
            \left(\mu^{-1}\tilde{\Delta}_i^+\mu\right)(\hat{V}_k,\hat{V}_k)=n-i
        \end{equation}
    \end{corollary}
    In order to determine the asymptotics of the eigenvalues one needs to determine the asymptotics of each of these entries:
    \begin{proposition}
        For $q\rightarrow\infty$ the factor $\frac{\prod_{j\in V_k\backslash V_{k-1}}q^{\{j'\in V_{k+1}\backslash V_k|j'<j\}}}{\gBinom{|V_{k+1}\backslash V_{k-1}|}{|V_{k}\backslash V_{k-1}|}_q}\rightarrow 1$ if every $j\in V_k\backslash V_{k-1}$ dominates $V_{k+1}\backslash V_k$ i.e. $j>j'$ for $j'\in V_{k+1}\backslash V_k$ and otherwise $\frac{\prod_{j\in V_k\backslash V_{k-1}}q^{\{j'\in V_{k+1}\backslash V_k|j'<j\}}}{\gBinom{|V_{k+1}\backslash V_{k-1}|}{|V_{k}\backslash V_{k-1}|}_q}\rightarrow 0$
        \label{prop:DegenerateEntries}
    \end{proposition}
    \begin{proof}
        One first notes that $\gBinom{a}{b}_q$ is a polynomial of degree $(a-b)b$ in $q$ with leading coefficient $1$. This means that the denominator is a polynomial in $q$ of order $|V_k\backslash V_{k-1}||V_{k+1}\backslash V_k|$. By construction the degree of the enumerator is bounded by $|V_k\backslash V_{k-1}||V_{k+1}\backslash V_k|$ and the degree is equal $|V_k\backslash V_{k-1}||V_{k+1}\backslash V_k|$ if and only if every $j\in V_k\backslash V_{k-1}$ dominates $V_{k+1}\backslash V_k$. As $q\rightarrow\infty$ this proves the proposition.
    \end{proof}
    We will refer to the condition for the fraction above to converge to $1$ as the \textbf{domination condition}. Since any entry in $\mu^{-1}\tilde{\Delta}_i^+\mu$ is the product of two such fractions we have that $\mu^{-1}\tilde{\Delta}_i^+\mu(\hat{V}_k, \hat{V}_l)$ does not vanish in the limit only if both $\hat{V}_k$ and $\hat{V}_l$ satisfy the domination condition. 
    \subsection{The Block Decomposition}
    \label{sec:TheBlock}
    From now on let $D_i$ denote the limit of $\mu^{-1}\tilde{\Delta}_i^+\mu$.
    One should think of $D_i$ as the adjacency matrix of a graph $Y^i$ with the vertices corresponding to all the $i$-flags of subsets of $[n]$ and the edges being $\{(x,y)|D_i(x,y)=\pm1\}$. Towards calculating the spectrum of $D_i$ it is useful to first determine the connected components of $Y^i$. When the indicators of vertices of connected component are bunched together, $D_i$ becomes block-diagonal and each  block corresponds to a component.\\
    
     As a warm up consider the following: Let $V$  be a flag $V_0\subset\dots\subset V_i\subset V_{i+1}=[n]$ such that for $j\neq k$ the set $V_j$ is an interval ending with $n$ i.e. $V_j=[n]\backslash[m_j]$ for some $m_j<n$. Now, let $\hat{V}_k\subset V$ be be the sub-flag of $V$ missing $V_k$. For $D_i(\hat{V}_k, \hat{V}_l)$ to not be $0$ one needs according to Proposition \ref{prop:DegenerateEntries} that every $j\in V_k\backslash V_{k-1}$ dominates $V_{k+1}\backslash V_k$. This implies that $V_k=[n]\backslash[m_k]$ for some $m_{k+1}<m_k<m_{k-1}$. Therefore, $V$ is a flag of intervals ending at $n$. In this case for any $l$ the domination condition $\forall j\in V_l\backslash V_{l-1},j'\in V_{l+1}\backslash V_l\implies j>j'$ holds. This means that if $V$ is a flag of intervals ending with $n$ then $D_i(\hat{V}_k,\hat{V}_l)=\pm1$\\
    As a corollary of the preceding discussion one has:
    \begin{proposition}
        The flags consisting of intervals starting with $n$ form a connected component of $Y^i$.
        \label{prop:EmptyCharacter}
    \end{proposition}
    \begin{proof}
        The discussion shows that from a flag of intervals ending with $n$ one can only transition to another flag of intervals ending with $n$. On the other hand the transition $D_i(\hat{V}_k,\hat{V}_l)$ corresponds to adding an arbitrary interval $[n]\backslash[m_k]$ where $m_{k-1}>m_k>m_{k+1}$ and then erasing the $l$-th interval. In that way in order to get from $m=(m_0>\dots>m_i)$ to $m'=(m_0'>\dots>m_i')$ one may successively add the $m_j'$ which are not in the flag $m$ and delete those $m_j$ which are not in $m'$ i.e. $D_i$ is transitive on the flags of intervals. 
    \end{proof}   
    To determine all connected components of $Y^i$ observe the following: Let $V$ be a flag $V_0\subset\dots\subset V_i\subset V_{i+1}=[n]$. For $V$ we call the sub-flag $\mathbf{V}=(V_{k_0}\subset\dots\subset V_{k_\alpha})$ the \textbf{characteristic} sub-flag which consists of all  $V_k\in V$ such that $\exists j\in V_k\backslash V_{k-1}\exists j'\in V_{k+1}\backslash V_k$ with $j<j'$ meaning that for the pair $V_k,V_{k+1}$ the domination condition is not satisfied. For flags consisting of intervals starting with $n$ the characteristic flag is empty. 
    \begin{proposition}
        The space of characteristic flags indexes the connected components of $Y_i$ i.e. each set of vertices of $Y_i$ with the same characteristic flag forms a connected component.
        \label{prop:BlockIntervals}
    \end{proposition}
    \begin{proof}
        We need to show two things, on one hand that for two neighbors in $Y_i$ the characteristic flag is the same and on the other hand that all flags with the same characteristic flag lie in the same connected component.\\ 
        Let $V$ be an $i+1$ flag and $k,l$ such that $D_i(\hat{V}_k,\hat{V}_l)=\pm 1$. Let $\mathbf{V}$ be the characteristic flag of $V$ and $\hat{\mathbf{V}}_k$ the characteristic flag of  $\hat{V}_k$. $D_i(\hat{V}_k,\hat{V}_l)=\pm 1$ so the elements in $V_k\backslash V_{k-1}$ dominate those in $V_{k+1}\backslash V_k$. This shows that $V_k$ is not in $\mathbf{V}$. Let $m$ be the smallest element in $V_{k-1}\backslash V_{k-2}$ and $m'$ be the largest element in $V_{k+1}\backslash V_{k-1}$. Now $m'$ must also be in $V_k\backslash V_{k-1}$. If $m>m'$ then $V_k\backslash V_{k-1}$ dominates $V_{k-1}\backslash V_{k-2}$ and equivalently $V_{k+1}\backslash V_{k-1}$ dominates $V_{k-1}\backslash V_{k-2}$, i.e. $V_{k-1}\in \hat{\mathbf{V}}_k\Leftrightarrow V_{k-1}\in \mathbf{V}$. Similarly one also shows that $V_{k+1}\in \hat{\mathbf{V}}_k\Leftrightarrow V_{k+1}\in \mathbf{V}$. In conclusion one has $\hat{\mathbf{V}}_k=\mathbf{V}$. Running the same argument replacing $k$ by $l$  also   $\mathbf{V}=\hat{\mathbf{V}}_l$ holds. As a consequence $D_i$ preserves the characteristic flag.\\
        Given a characteristic flag $\mathbf{V}$ transitivity of $D_i$ on the set of flags with characteristic flag $\mathbf{V}$ the proof works similarly to the the interval case: for $V_k\in V$ let  $\mathbf{V}_j$ be the largest member of the characteristic flag of $V$ such that $\mathbf{V}_j\subset V_k$ then one can write $V_k=\mathbf{V}_j\cup(\mathbf{V}_{j+1}\cap ([n]\backslash[m_k])$ for an appropriate $m_k$. To walk between two flags with the same characteristic subflag one may add and delete the sets of the form $V_k=\mathbf{V}_j\cup(\mathbf{V}_{j+1}\cap ([n]\backslash[m_k])$ until the two flags match i.e. for two flags $V_1$ and $V_2$ with characteristic flag $\mathbf{V}$ there is a path in $Y_i$ connecting the two.
    \end{proof}
    In order to determine the spectrum of $D_i$ one needs to determine the spectrum of $D_i$ restricted to each connected component. To get a clear picture of what $D_i$ restricted to each block looks like, one may again first consider the block corresponding to the trivial characteristic flag:\\
    Here each vertex is characterized by a sequence of integers: $\hat{V}_k=n\geq m_0>\dots>\hat{m}_k>\dots>m_{i+1}>1$. Note that $m_{i+1}>1$ since $span(\{e_i|n\geq i\geq 1\}=\mathbb{F}_q^n$ is the trivial subspace i.e. not a vertex of the of the building. Further the sign of $D_i(\hat{V}_k,\hat{V}_l)=(-1)^{k+l}$. Putting these two observations together, then one notes that $D_i$ restricted to the block of vertices with trivial characteristic flag is a model for the signed up-down walk of level $i$ on the $n-2$-dimensional complete complex. A similar statement also holds for any other connected component of $Y^i$.
    \begin{proposition}
        Let $\mathbf{V}=(\mathbf{V}_1\subset\dots\subset\mathbf{V}_k)$ be a characteristic flag. Then the restriction of $D_i$ to the connected component of $Y_i$ associated to $\mathbf{V}$ is equivalent to the signed-up down walk of level $i-k$ on a $n-2-k$-dimensional complete complex.
        \label{prop:BlockSimplex}
    \end{proposition}
    \begin{proof}
        For every flag $V$ with characteristic flag $\mathbf{V}$ every $V_j\in V$ with $\mathbf{V}_l\subset V\subset \mathbf{V}_{l+1}$ can be described as $\mathbf{V}_l\cup(\mathbf{V}_{l+1}\cap [n]\backslash[m_j])$ for some $1\leq m_j\leq n$. Therefore, $D_i\restriction_{\mathbf{V}}$ is equivalent to the signed-up down walk on the $n-2$-dimensional complete complex while enforcing that the flags contain $|\mathbf{V}_1|<\dots<|\mathbf{V}_k|$. In other words, $D_i\restriction_{\mathbf{V}}$ is as a matrix equivalent to the minor of the up-down walk on the $n-2$ dimensional complete complex corresponding to sub-simplices containing $|\mathbf{V}_1|<\dots<|\mathbf{V}_k|$.\\
        As shown in Section \ref{sec:WalkInvariance} the up-down walk is invariant under automorphisms of the underlying complex. For the $n-2$-dimensional complete complex the group of automorphisms is the symmetric group $S_{n-1}$. Let $g\in S_{n-1}$ be a permutation which maps $\{|\mathbf{V}_1|,\dots,|\mathbf{V}_k|\}$ to $\{n-k,\dots,n-1\}$. After applying $g$ the minor corresponding to simplices containing $\{|\mathbf{V}_1|,\dots,|\mathbf{V}_k|\}$ becomes the minor corresponding to simplices  containing $\{n-k,\dots,n-1\}$. But this minor is equivalent to the signed up down walk of level $i-k$ on the $n-k-2$ dimensional complete complex by mapping the vertex with label $j<n-k$ in the simplex of dimension $n-2$ to the vertex of label $j$ in the simplex of dimension $n-k-2$. In conclusion this shows that $D_i\restriction_{\mathbf{V}}$ is equivalent to the signed up down walk of level $i-k$ on the simplex of dimension $n-2-k$.
    \end{proof}
    \subsection{The Spectrum of a Simplex}
    \label{sec:SimplexSpecs}
    In the Section \ref{sec:TheBlock} we have shown that $D_i$ can be block decomposed such that each block corresponds to a characteristic flag $\mathbf{V}$. Further, the restriction of $D_i\restriction_{\mathbf{V}}$ is equivalent to the signed up-down walk of level $i-|\mathbf{V}|$ on the $n-2-|\mathbf{V}|$-dimensional complete complex. \\
    In the Section \ref{sec:WalkInvariance} we showed that $\Delta_i^+$ is invariant under automorphisms of the underlying simplicial complex. For the $n$-dimensional complete complex $Z_n$ the group of symmetries is $S_{n+1}$ i.e. the symmetric group on $n+1$ elements. Clearly $S_{n+1}$ acts transitively on the set of its $i$-simplices $Z_n(i)$. Let $s=\{0,\dots, i\}$ and $H_s$ its orientation preserving stabilizer i.e. for $h\in H_s$ the sign of $h$ restricted to $s$ is $sign(h\restriction_s)=+1$. Note that $H_s$ is isomorphic to $A_{i+1}\times S_{n-i}$ where the alternating group on $i+1$-elements $A_{i+1}$ acts on the first $i+1$ elements and $S_{n-i}$ acts on the tail. For $A_{i+1}$ one has
    \begin{lemma}
        $A_{i+1}$ acts  transitively on the subsets of fixed cardinality $\{0,\dots, i\}$ for all $i\neq 1$ i.e. for any two subsets $t,u\subset \{0,\dots, i\}$ with $|t|=|u|$ there is $h\in A_{i+1}$ with $ht=u$.
    \end{lemma}
    This leads to the following lemma:
    \begin{lemma}
        Assume $i\neq 1$. If $s\in Z_n(i)$ then for each $0\leq k\leq i+1$ the simplices with $t\in Z_n(i)$ such that $|t\cap s|=k$ from an $H_s$ orbit. 
    \end{lemma}
    \begin{proof}
        Since $H_s=A_{i+1}\times S_{n-i}$ one has for $h\in H_s$ that $|t\cap s|=|ht\cap s|$. By the previous lemma on transitivity for two simplices $t,t'$ with $|t\cap s|=|t'\cap s|$ there is  a $h\in A_{i+1}$ such that $ht\cap s=t'\cap s$. Further, $S_{n-i}$ acts set transitively by cardinality as well so there is $h'\in S_{n-i}$ such that $h't\cap\{i+1\dots n\}=t'\cap\{i+1\dots n\}$.
    \end{proof}
    As $S_{n+1}$ acts transitively on $Z_n(i)$ in order to determine the spectrum of the $i$-th signed up-down walk on $Z_n$ one only needs to consider its spectrum restricted to $H_s$-invariant functions:
    \begin{proposition}
        Let $f\neq 0$ be a an eigenfunction of the signed up-down walk on $Z_n$ corresponding to eigenvalue $\lambda$, then there is an $H_s$-invariant eigenfunction $f_s$ which corresponds to $\lambda$ as well.
    \end{proposition}
    \begin{proof}
        Let $s'$ be a simplex such that $f(s')\neq 0$ and let $g\in S_{n+1}$ be such that $gs'=s$. Then $gf(s)=f(s')\neq 0$. Set:
        \begin{equation}
            f_s=\frac{1}{|H_s|}\sum_{h\in H_s}hgf
        \end{equation}
        Then $f_s(s)=f(s')\neq 0$, is an eigenfunction of the signed up-down walk corresponding to $\lambda$ and $H_s$-invariant by construction.
    \end{proof}
    Luckily, the space of $H_s$ invariant functions has exactly two dimensions in most cases:
    \begin{proposition}
        An $H_s$-invariant function vanishes on any simplex which has overlap less than $i$ with $s$.
    \end{proposition}
    \begin{proof}
        If the overlap of $t$ and $s$ is strictly less than $i$ there are at least two distinct elements $k, l\in t\backslash s$. As the cycle $(kl)$ fixes each element in $s$ it stabilizes $s$ i.e. $(kl)\in H_s$. Further $sign((kl)\restriction_t)=-1$ which implies:
        \begin{equation}
            \frac{1}{|H_s|}\sum_{h\in H_s}hf_t=\frac{1}{2}\left(\frac{1}{|H_s|}\sum_{h\in H_s}hf_t+\frac{1}{|H_s|}\sum_{h\in H_s}h(kl)f_t\right)=\frac{1}{2}\left(\frac{1}{|H_s|}\sum_{h\in H_s}hf_t-\frac{1}{|H_s|}\sum_{h\in H_s}hf_t\right)=0
        \end{equation}
        For an $H_s$ invariant function $f$ one has $\frac{1}{|H_s|}\sum_{h\in H_s}hf=f$ which implies by the previous equation that $f(t)=0$.
    \end{proof}
    Using this vanishing result one may easily determine the spectrum of the signed up-down walk on any simplex:
    \begin{proposition}
        The spectrum of the $i$-th signed up-down walk on the $n$-dimensional complete complex is given as $\{0,n+1\}$
        \label{prop:UpDownSimplex}
    \end{proposition}
    \begin{proof}
        According to the previous proposition one needs to determine the spectrum of the $2\times 2$-matrix corresponding to the walk matrix restricted to $H_s$-invariant functions. One chooses as a basis the indicator of $f_s$ and $\sum_{\hat{s}_k\cup\{l\}}(-1)^{k+i+1}f_{\hat{s}_k\cup\{l\}}$.\\
        $s$ has $(i+1)(n-i)$ neighbors $t$ with $|s\cap t|=i$ and by our choice of signs each edge counts positively. Each such $t$ has $1$ edge connecting it to $s$. This determines the off-diagonal entries of $\Delta_i^+\restriction_{H_s}$. For the entries on the diagonal, one needs to count the $n-i$ loops attached to each vertex positively.
        One can transition from $\hat{s}_k\cup\{l\}$ to $\hat{s}_k\cup\{l'\}$ with $l'$, there are $n-i-1$ such neighbors and each edge counts as $(-1)^{2i+1}=-1$. And lastly, one may transition from $\hat{s}_k\cup\{l\}$ to $\hat{s}_{k'}\cup\{l\}$ which contributes a $(-1)^{k+k'}=(-1)^{k+i+1}(-1)^{k'+i+1}$ i.e. a positive sign. In total one obtains:
        \begin{equation}
            \begin{pmatrix}
                n-i+i-(n-i-1)&&1\\
                (i+1)(n-i)&&n-i              
            \end{pmatrix}=\begin{pmatrix}
                i+1&&1\\
                (i+1)(n-i)&&n-i              
            \end{pmatrix}
        \end{equation}
        For the case $(n,i)=(n,2)$ consider the following: There are three relevant $H_s$ orbits $s$, $\{(0,j)|2\leq j\leq n\}$ and $\{(1,j)|2\leq j\leq n\}$. By the same considerations as in the case $i\neq 2$ the signed up-down operator restricted to $H_s$-invariant functions reads:
        \begin{equation}
            \begin{pmatrix}
                1&&1&&1\\
                1&&1&&1\\
                n-1&&n-1&&n-1
            \end{pmatrix}
        \end{equation}
        where the $e_1$ and $e_2$ is the indicator of $\{(0,j)|2\leq j\leq n\}$ and $\{(1,j)|2\leq j\leq n\}$ respectively while $e_3$ is the indicator of $s$. The eigenvalues of this matrix are $\{0,n+1\}$ as required. 
    \end{proof}
    \subsection{The Asymptotic Eigenvalues}
    In  order to obtain a full proof of Papikian's conjecture one needs to combine the block decomposition derived in Section \ref{sec:TheBlock} with the spectrum of the signed up-down walk on the $n$-dimensional complete complex determined in Section \ref{sec:SimplexSpecs}: 
    \begin{corollary}
        The spectrum of the signed up-down walk $\Delta_i^+$ on $X_{n,q}$ converges to $\{0,n-i-2,\dots,n-1\}$ for $q\rightarrow\infty $. 
        \label{thm:Eigenstripping}
    \end{corollary}
    \begin{proof}
        Lemma \ref{lem:SpectralConvergence} together with Proposition \ref{prop:SymMatrix} show $spec(\Delta_i^+)=spec(\tilde{\Delta}_i^+)\rightarrow spec(D_i)$. Proposition \ref{prop:BlockIntervals} show that each block of $D_i$ corresponds to a characteristic flag $\mathbf{V}$ while Proposition \ref{prop:BlockSimplex} establishes the equivalence of $D_i\restriction_{\mathbf{V}}$ with the signed up-down walk of level $i-|\mathbf{V}|$ on the $n-2-|\mathbf{V}|$ dimensional complete complex. Proposition \ref{prop:UpDownSimplex} shows that this spectrum is exactly $\{0, n-1-|\mathbf{V}|\}$. So what is left to do is determining what lengths a characteristic flag can have. Proposition \ref{prop:EmptyCharacter} gives an example of $|\mathbf{V}|=0$ and naturally $|\mathbf{V}|\leq i+1$. For $0\leq j\leq i$ setting $V=[0]\subset[1]\subset\dots\subset[j]\subset [j]\cup[n]\backslash[n-1]\subset [j]\cup[n]\backslash[n-2]\subset\dots $ the characteristic flag is $\mathbf{V}=([0]\subset\dots\subset[j])$ i.e. $V$ is an $i$ flag with characteristic flag of length $j+1$. In conclusion this shows that
        \begin{equation}
            spec(D_i)=\bigcup_{\mathbf{V}}spec(D_i\restriction_{\mathbf{V}})=\bigcup_{j=0}^{i+1}\{0,n-1-j\}
        \end{equation}
    \end{proof}

    \newpage
    \section{Appendix}
    \subsection{Counting Subspaces}
    \label{sec:CountingSubspaces}
    The complexes we are interested in are defined based on $\mathbb{F}_q^n$ namely through the notion of a \textbf{flag} of that space. A  flag is a collection of distinct non-trivial subspaces $V_i\subset\mathbb{F}_q^n$ i.e. $V_i\not \in \{\emptyset,\mathbb{F}_q^n\}$ and $V_i\subset V_{i+1}$. The flag is called \textbf{full} if this collection has size $n-1$ and otherwise it is called \textbf{partial}.\\
	Throughout this note we often need to count neighbors and to this end \textbf{Gaussian binomial coefficients} prove very useful. These coefficients express for integers $k\leq n$ and prime-power $q$, denoted as $\gBinom{n}{k}_q$, how many $k$-dimension subspaces there are in $\mathbb{F}_q^n$. 
	\begin{lemma}
		The Gaussian binomial coefficients can be expressed as follows:
		\begin{equation}
			\gBinom{n}{k}_q=\prod_{i=1}^{k}\frac{q^{n-i+1}-1}{q^i-1}
		\end{equation}
		\label{lem:GaussBinom}
	\end{lemma}
	\begin{proof}
		As $\gBinom{n}{1}_q$ simply counts how many lines there are in $\mathbb{F}_q^n$ and there are $q^n-1$ vectors which span a line (as only $0$ does not span a line) and $q-1$ vectors span the same line one has
		\begin{equation}
			\gBinom{n}{1}_q=\frac{q^n-1}{q-1}
		\end{equation}
		Now by induction assume that the formula above is established for $\gBinom{n}{k-1}_q$. Choosing a $k$-dimensional subspace is equivalent to choosing a $k-1$-dimensional subspace and then adding one dimension by choosing a line in the $n-k+1$-dimensional quotient.  For a fixed $k$ dimensional $V$ space there are $\gBinom{k}{k-1}_q$ pairs of $k-1$-dimensional subspaces $V'$ and lines $L\subset \mathbb{F}_q^n/V'$ which lead to $V$. Therefore this method leads to overcounting by a factor of $\gBinom{k}{k-1}_q$ and one has $\gBinom{n}{k}_q=\gBinom{n}{k-1}_q\gBinom{n-k+1}{1}_q/\gBinom{k}{k-1}_q$. Lastly, one notes that while fixing a scalar product on $\mathbb{F}_q^n$ mapping a subspace to its dual space induces a bijective map between $k$ subspaces and $n-k$ subspaces and in particular $\gBinom{k}{k-1}_q=\gBinom{k}{1}$. Altogether this proves the lemma. 
	\end{proof}
	In Section \ref{sec:EdgeWeights} one needs to count the number of full flags of the vector space $\mathbb{F}_q^m$> 
	\begin{proposition}
		The number of full flags of $\mathbb{F}_q^m$ is given as
		\begin{equation}
			\prod_{i=0}^{m}\gBinom{m-j}{1}_q
		\end{equation}
		\label{prop:FullFlags}
	\end{proposition}
	\begin{proof}
		A full flag of $\mathbb{F}_q^m$ can be constructed inductively by first choosing a subspace $V_1$ of dimension $1$, then $V_2$ of dimension $2$ such that $V_1\subset V_2$ and so on. Observe that choosing $V_{i+1}$ given $V_i$ is equivalent to choosing a one dimensional subspace $[V_{i+1}]$ in the quotient $\mathbb{F}_q^m/V_i$. Since the dimension of $\mathbb{F}_q^m/V_i$ is $m-i$ one has 
        \begin{equation}
            \gBinom{m-i}{1}_q
        \end{equation}
        possibilities for choosing $V_{i+1}$. In each step the choice is independent so for choosing a full flag one has $\prod_{i=0}^{m}\gBinom{m-j}{1}_q$-many options. 
	\end{proof}
    \newpage
	\bibliography{Bibliography}
	\bibliographystyle{alpha}
\end{document}